\documentclass[a4,12pt]{amsart}

\usepackage{amssymb}
\usepackage{amstext}
\usepackage{amsmath}
\usepackage{amscd}
\usepackage{amsthm}
\usepackage{amsfonts}
\usepackage{enumerate}
\usepackage{graphicx}
\usepackage{latexsym}
\usepackage[all]{xy}

\newtheorem*{maintheorem*}{Main Theorem}
\newtheorem*{corollary*}{Corollary}
\newtheorem{theorem}{Theorem}[section]
\newtheorem{corollary}[theorem]{Corollary}
\newtheorem{lemma}[theorem]{Lemma}
\newtheorem{proposition}[theorem]{Proposition}
\newtheorem*{proposition*}{Proposition}
\newtheorem{question}[theorem]{Question}
\newtheorem*{question*}{Question}

\newtheorem*{claim*}{Claim}

\theoremstyle{definition}
\newtheorem{definition}[theorem]{Definition}
\newtheorem{remark}[theorem]{Remark}

\newtheorem{definitiontheorem}[theorem]{Definition-Theorem}

\theoremstyle{remark}

\newtheorem{condition}[theorem]{Condition}

\numberwithin{equation}{theorem}

\renewcommand{\mod}{\operatorname{mod}}
\newcommand{\proj}{\operatorname{proj}}

\newcommand{\End}{\operatorname{End}}
\newcommand{\Hom}{\operatorname{Hom}}
\newcommand{\add}{\operatorname{\mathsf{add}}}
\newcommand{\Ext}{\operatorname{Ext}}

\newcommand{\Tr}{\operatorname{Tr}}

\newcommand{\Db}{\mathsf{D}^{\rm b}}
\newcommand{\Kb}{\mathsf{K}^{\rm b}}

\newcommand{\T}{\mathbb{T}}
\newcommand{\C}{\mathcal{C}}
\renewcommand{\P}{\mathbb{P}}
\newcommand{\ORA}{\overrightarrow}
\renewcommand{\H}{\mathcal{H}}

\newcommand{\silt}{\operatorname{\mathsf{silt}}}
\newcommand{\tsilt}{\operatorname{\mathsf{2silt}}}

\newcommand{\sttilt}{\operatorname{\mathsf{s\tau -tilt}}}
\newcommand{\stitilt}{\operatorname{\mathsf{s\tau^{-} -tilt}}}
\newcommand{\thick}{\operatorname{\mathsf{thick}}}
\newcommand{\trigid}{\operatorname{\mathsf{\tau\text{-}rigid}}}
\newcommand{\supp}{\operatorname{Supp}}
\newcommand{\fac}{\operatorname{Fac}}
\newcommand{\dpr}{\operatorname{dp}}
\newcommand{\Z}{{\mathbb{Z}}}

\renewcommand{\P}{\mathbb{P}}

\newcommand{\Rad}{\operatorname{rad}}

\newcommand{\old}[1]{{\color{red}#1}}

\begin{document}
\title[Sharing the support $\tau$-tilting poset with tree quiver algebras]
{Algebras sharing the same support $\tau$-tilting poset with tree quiver algebras}

\author{Takuma Aihara}
\address{Department of Mathematics, Tokyo Gakugei University,
4-1-1 Nukuikita-machi, Koganei, Tokyo 184-8501, Japan}
\email{aihara@u-gakugei.ac.jp}

\author{Ryoichi Kase}
\address{Department of mathematics, Nara Women's University, Kitauoya-Nishimachi, Nara city, Nara 630-8506, Japan}
\email{r-kase@cc.nara-wu.ac.jp}

\thanks{The first author was supported by Grant-in-Aid for Young Scientists 15K17516.}

\begin{abstract}
Happel and Unger reconstructed hereditary algebras from their posets of tilting modules.
Inspired by this result, we try removing the assumption to be hereditary.
However, it would be unfortunately fail in general: e.g. every selfinjective algebra has the poset consisting of only one point.
Therefore, we should consider a generalization of the Happel-Unger's result for posets of support $\tau$-tilting modules,
which contains those of tilting modules.
In this paper, we spotlight finite dimensional algebras
whose support $\tau$-tilting posets coincide with those of tree quiver algebras and give a full characterization of such algebras.
\end{abstract}

\maketitle

\section*{Introduction}

In 1970's, tilting theory appeared to describe the transition of the module category (and derived category) structure of two (finite dimensional) algebras by using tilting modules.
It turned out that the class of such modules was very crucial, because they induces derived equivalences between algebras \cite{H}.
Therefore, we naturally ask to obtain many tilting modules.

One of approaches to get tilting modules is to use tilting mutation, which is an operation to construct a new tilting module from a given one by replacing a direct summand.
An original prototype of tilting mutation is in the notion of APR/BB tilting modules \cite{APR, BB}, which were formulated as tilting mutation for tilting modules by Riedtmann and Schofield \cite{RS}.
They also introduced the notion of tilting quivers to investigate the behavior of tilting modules by tilting mutation.
Interestingly enough, the set of tilting modules has poset structure, say the tilting poset, and its Hasse quiver coincides with its tilting quiver \cite{HU1, HU2}.
Here, we should remark that tilting mutation for tilting modules is not necessarily possible: Whether or not we can get a tilting module by tilting mutation depends on the choice of a direct summand of a given tilting module.

To provide this disadvantage, a generalization of tilting modules, so-called \emph{support $\tau$-tilting modules}, was introduced by Adachi, Iyama and Reiten \cite{AIR},
who moreover innovated \emph{support $\tau$-tilting mutation};
the quiver obtained by support $\tau$-tilting mutation is called the \emph{support $\tau$-tilting quiver}.
A point worthy of special mention is that support $\tau$-tilting mutation is always possible.
Furthermore, it was shown that the set of support $\tau$-tilting modules also has poset structure and the Hasse quiver coincides with its support $\tau$-tilting quiver;
We call the poset the \emph{support $\tau$-tilting poset}.

It is trivial that the structure of the tilting and the support $\tau$-tilting posets is determined by a given algebra.
So, we have an interesting question:

\begin{question*}
Can we reconstruct an algebra from the structure of the (support $\tau$-) tilting poset?
\end{question*}
We will focus on support $\tau$-tilting posets, but not support $\tau$-tilting quivers:
It would be very difficult to reconstitute an alebra by the support $\tau$-tilting quiver.
Moreover, it is still a open problem to give two algebras which have the same support $\tau$-tilting quiver and the different support $\tau$-tilting poset.

Happel and Unger gave a fascinating answer to this question,
that is, any hereditary algebra without multiple arrows in the (Gabriel) quiver is uniquely determined (up to isomorphism) by the tilting poset \cite{HU3}.

In this paper, we consider algebras admitting the same support $\tau$-tilting poset with a tree quiver algebra.
Now, we state a main theorem of this paper.

\begin{maintheorem*}[Corollary \ref{fc}]
Let $\Lambda:=kQ/I$ be a finite dimensional algebra over an algebraically closed field $k$ presented by a quiver $Q$ and an admissible ideal $I$ of $kQ$.
Then $\Lambda$ has the same support $\tau$-tilting poset with a tree quiver algebra if and only if 
\begin{itemize}
\item The quiver $Q^\circ$ obtained by removing all loops from $Q$ is a tree;
\item $e_i\Lambda\alpha=\alpha\Lambda e_j$ for any arrow $\alpha:i\to j\ (i\neq j)$ of $Q$, where $e_i$ stands for the primitive idempotent of $\Lambda$ corresponding to a vertex $i$ of $Q$;
\item Any path of $Q^\circ$ does not belong to $I$.
\end{itemize}
\end{maintheorem*}


\subsection*{Notation} Throughout this paper, let $\Lambda=kQ/I$ be a basic finite dimensional algebra over an algebraically
closed field $k$, where $Q$ is a finite quiver and $I$ is an admissible ideal of $kQ$.

\begin{enumerate}[(1)]
\item We denote by $Q_0$ and $Q_1$ the sets of vertices and arrows of $Q$, respectively.
For arrows $\alpha:a_0\to a_1$ and $\beta: b_0\to b_1$ of $Q$,
we mean by $\alpha\beta$ the path $a_0\xrightarrow{\alpha}a_1\xrightarrow{\beta}b_1$ if $a_1=b_0$, otherwise 0 in $kQ$.

\item We denote by $\mod \Lambda\ (\proj\Lambda)$ the category of finitely generated (projective) right $\Lambda$-modules.

\item By a module, we always mean a finitely generated right module.

\item Let $(\P, \leq)$ be a poset.
The interval of two elements $a,b$ of $\P$ with $a\leq b$ is denoted by $[a,b]:=\{x\in \P\ |\ a\leq x\leq b\}$.
We write the Hasse quiver of $\P$ by $\H(\P)$.
Let $\P'$ be a subset of $\P$.
When $\P'$ is a poset itself with a partial order $\underset{\P'}{\leq}$, we call it a \emph{subposet} of $\P$ if $a\underset{\P'}{\leq} b$ implies $a\leq b$ for any elements $a, b$ of $\P'$.
A subposet $\P'$ of $\P$ is said to be \emph{full} provided the converse of the implication above holds:
i.e. the partial orders of $\P$ and $\P'$ coincide.
Moreover, we say that a subposet $\P'$ of $\P$ is \emph{induced} if it is full and the partial order $\leq$ induces that $\H(\P')$ is a full subquiver of $\H(\P)$.
%
%
\end{enumerate} 
\section{Preliminary}

This section is devoted to recalling the definitions and their basic properties of 
support $\tau$-tilting modules and silting complexes.


\subsection{Support $\tau$-tilting modules}
\label{subsec:2.2}
For a module $M$, we denote by $|M|$ the number of non-isomorphic indecomposable direct summands of $M$.
The Auslander-Reiten translation is denoted by $\tau$. 

Let us recall the definition of support $\tau$-tilting modules.

\begin{definition}
Let $M$ be a $\Lambda$-module and $P$ a projective $\Lambda$-module.
\begin{enumerate}
\item We say that $M$ is $\tau$-\emph{rigid} provided it satisfies
$\Hom_{\Lambda}(M, \tau M)=0$.
\item A pair $(M,P)$ is also said to be  $\tau$-\emph{rigid} if $M$ is $\tau$-rigid and $\Hom_\Lambda(P,M)=0$.
\item 
A support $\tau$-tilting pair $(M,P)$ is defined to be a $\tau$-rigid pair with $|M|+|P|=|\Lambda|$. 
\end{enumerate}

By \cite[Proposition 2.3]{AIR}, a support $\tau$-tilting pair $(M,P)$ is uniquely determined by the module $M$.
Thus, we often use only $M$ instead of $(M,P)$, and call it a \emph{support $\tau$-tilting module}.

The set of isomorphism classes of basic support $\tau$-tilting modules of $\Lambda$ is denoted by $\sttilt\Lambda$.
\end{definition}

Dually, we define support $\tau^-$-tilting modules and their set $\stitilt\Lambda$.
Then, one can easily check that $M$ belongs to $\sttilt\Lambda$
if and only if $DM$ is in $\stitilt\Lambda^{\mathrm{op}}$,
where $D$ and $\Lambda^{\mathrm{op}}$ stand for the $k$-dual
and the opposite algebra of $\Lambda$.
(See \cite{AIR} for the details.)

We observe important properties.

\begin{proposition}\cite[Proposition 1.3, Lemma 2.1]{AIR}
\label{basicfact} The following hold.
 \begin{enumerate}[(1)]
\item A $\tau$-rigid pair $(M,P)$ satisfies the inequality $|M|+|P|\leq |\Lambda|$. 
\item Let $J$ be an ideal of $\Lambda$. Let $M$ and $N$ be  $(\Lambda/J)$-modules .
If $\Hom_{\Lambda}(M,\tau N)=0$, then $\Hom_{\Lambda/J}(M,\tau_{\Lambda/J} N)=0$.
Moreover, if $J=(e)$ is an two-sided ideal generated by an idempotent $e$, then the converse holds.

\end{enumerate}
\end{proposition}

We denote by $e_i$ the corresponding primitive idempotent to a vertex $i$ of $Q$.

Let $M$ be a module.
We define a subset of $Q_0$ by
\[\supp(M):=\{i\in Q_0\ |\ Me_i\neq0 \}.\]
If $(M, e\Lambda)$ is a support $\tau$-tilting pair for some idempotent $e$, 
then $\supp(M)$ coincides with the set of vertices $i$ satisfying $ee_i=0$.

Denote by $\fac M$ the full subcategory of $\mod\Lambda$ consisting of factor modules of finite direct sums of copies of $M$.

We introduce a partial order on $\sttilt\Lambda$.

\begin{definitiontheorem}\cite[Lemma 2.25]{AIR}
For support $\tau$-tilting modules $M$ and $M'$, we write $M\geq M' $ if $\fac M\supseteq \fac M'$.
 Then one has the following equivalent conditions:
\begin{enumerate}
\item $M\geq M'$;
\item  $\Hom_{\Lambda}(M',\tau M)=0$ and $\supp(M)\supseteq \supp(M')$.
\end{enumerate}
Moreover, $\geq$ gives a partial order on $\sttilt\Lambda$.

The Hasse quiver of the poset $\sttilt\Lambda$ is denoted by $\H(\sttilt\Lambda)$.
\end{definitiontheorem} 

Let $(N,R)$ be a pair of a module $N$ and a projective module $R$.

We say that $(N,R)$ is \emph{basic} if so are $N$ and $R$.
A direct summand $(N', R')$ of $(N,R)$ is also a pair of a module $N'$ and a projective module $R'$ which are direct summands of $N$ and $R$, respectively.

A pair $(N,R)$ is said to be \emph{almost complete support $\tau$-tilting} provided it is a $\tau$-rigid pair with $|N|+|R|=|\Lambda|-1$.

We make an observation of $\H(\sttilt\Lambda)$.

\begin{theorem}\label{basicprop}
\begin{enumerate}[(1)]
\item \cite[Theorem 2.18]{AIR} 
Every basic almost complete support $\tau$-tilting pair is a direct summand of exactly two basic support $\tau$-tilting pairs.

\item \cite[Corollary 2.34]{AIR} Let $(M,P)$ and $(M',P')$ be basic support $\tau$-tilting pairs.
Then $M$ and $M'$ are connected by an arrow of $\H(\sttilt\Lambda)$ if and only if $(M,P)$ and $(M',P')$ have a common basic almost complete support $\tau$-tilting pair as a direct summand.
In particular, $\H(\sttilt\Lambda)$ is $|\Lambda|$-regular. 

\item \cite[Corollary 2.38]{AIR} If $\H(\sttilt\Lambda)$ has a finite connected component $\mathcal{C}$, then
$\mathcal{C}=\H(\sttilt\Lambda)$. 
\end{enumerate}   
\end{theorem}

%

For a basic $\tau$-rigid pair $(N,R)$, we define
\[\sttilt_R^N\Lambda:=\{M\in \sttilt\Lambda\ |\  N\in \add M,\Hom_{\Lambda}(R, M)=0\},\]
equivalently, which consists of all support $\tau$-tilting pairs having $(N,R)$ as a direct summand. 
For simplicity, we omit 0 if $N=0$ or $R=0$.

Given an idempotent $e=e_{i_1}+\cdots+e_{i_\ell}$ of $\Lambda$ so that $R=e\Lambda$,
we see that $M$ belongs to $\sttilt_R\Lambda$ if and only if it is a basic support $\tau$-tilting module
with $\supp(M)=Q_0\setminus\{i_1,\dots,i_\ell\}$.
Hence, by Proposition \ref{basicfact} this leads to a poset isomorphism $\sttilt_R\Lambda\simeq \sttilt\Lambda/(e)$.


In the rest of this subsection, we study the following question,
which is also of interest in this paper.

\begin{question}\label{questionlattice}
When does the poset $\sttilt\Lambda$ have lattice structure?
\end{question}

Here, a poset $\mathbb{P}$ is defined to be a \emph{lattice} if for any elements $a,b$ of $\mathbb{P}$, there is a maximum element $a\wedge b$ of $\{x\in \mathbb{P}\mid x\leq a,b \}$ and a minimum element $a\vee b$ of $\{x\in \mathbb{P}\mid x\geq a,b \}$.

Let $\Lambda$ be a hereditary algebra.
In this case, note that a (support) $\tau$-tilting module is nothing but a (support) tilting module. (See \cite{IT} for the definition of support tilting modules.) 

The following result gives us an answer to the question above.

\begin{theorem}\cite[Theorem 2.3]{IRTT}\label{IRTT} Assume that $\Lambda$ is hereditary.
Then the poset $\sttilt \Lambda$ has lattice structure if and only if $\Lambda$ is representation-finite or $|\Lambda|\leq 2$.
\end{theorem}

\subsection{Silting complexes}

We denote by $\Kb(\proj\Lambda)$ the bounded homotopy category of $\proj\Lambda$.

A complex $T=[\cdots\rightarrow T^i\rightarrow T^{i+1}\rightarrow\cdots]$ in $\Kb(\proj\Lambda)$ is said to be \emph{two-term} provided $T^i=0$ unless $i=0, -1$.

We recall the definition of silting complexes.

\begin{definition} Let $T$ be a complex in $\Kb(\proj \Lambda)$.
\begin{enumerate}
\item We say that $T$ is \emph{presilting} if $\Hom_{\Kb(\proj \Lambda)}(T,T[i])=0$ for any positive integer $i$.
\item A \emph{silting complex} is defined to be presilting and generate $\Kb(\proj \Lambda)$ by taking direct summands, mapping cones and shifts.   
\end{enumerate}
We denote by $\silt\Lambda\ (\tsilt \Lambda)$ the set of isomorphism classes of basic (two-term) silting complexes in $\Kb(\proj \Lambda)$.
\end{definition}

We give an easy property of (pre)silting complexes. 

\begin{lemma}\label{factformpp}\cite[Lemma 2.25]{AI}
Let $T=[0\to T^{-1}\to T^0\to 0]$ be a two-term presilting complex of $\Kb(\proj\Lambda)$. Then we have $\add T^{-1}\cap \add T^0=\{0\}.$

\end{lemma}

The set $\silt\Lambda$ also has poset structure as follows.

\begin{definitiontheorem}\cite[Theorem 2.11]{AI}
For silting complexes $T$ and $T'$ of $\Kb(\proj\Lambda)$, we write $T\geq T'$
 if $\Hom_{\Kb(\proj\Lambda)}(T, T'[i])=0$ for every positive integer $i$.
Then the relation $\geq$ gives a partial order on $\silt\Lambda$.
\end{definitiontheorem}

The following result connects silting theory with $\tau$-tilting theory.

\begin{theorem}\label{bijection}
\cite[Corollary 3.9]{AIR}
The assignment 
\[(M,P)\mapsto \left\{ \begin{array}{c}\xymatrix{
P_1 \ar[r]^{p_M} \ar@{}@<3pt>[r]^(0){(\mathrm{-1th})}\ar@{}@<3pt>[r]^(1){(\mathrm{0th})} \ar@{}[rd]|{\bigoplus} & P_0 \\
P \ar[r] & 0
} \end{array}\right.\]
where $p_M$ is a minimal projective presentation of $M$, gives rise to a poset isomorphism 
$\sttilt\Lambda\xrightarrow{\sim}\tsilt\Lambda$.

\end{theorem}

By this theorem, we will feel free to use silting complexes and support $\tau$-tilting modules.
Lemma \ref{factformpp} and Theorem \ref{bijection} say that every $\tau$-rigid module with a minimal projective presentation $P_1\to P_0$ satisfies $\add P_1\cap \add P_0=\{0\}$.

We will close this section by recalling the definition and an important property of g-vectors of complexes of $\Kb(\proj \Lambda)$.

Let $K_0(\proj\Lambda)$ be the Grothendieck group of $\proj\Lambda$ and $[P]$ denote the element in $K_0(\proj\Lambda)$ corresponding to a projective module $P$.
As is well-known, the set $\{[e_i\Lambda]\ |\ i\in Q_0 \}$ forms a basis of $K_0(\proj\Lambda)$.

\begin{definition}
Let $X=[P'\to P]$ be a two-term complex of $\Kb(\proj \Lambda)$ and write $[P]-[P']=\sum_{i\in Q_0}g_i^X [e_i \Lambda]$ in $K_0(\proj \Lambda)$ for some $g_i^X\in\Z$.
Then we call the vector $g^X:=(g_i^X)_{i\in Q_0}\in \Z^{Q_0}$ the \emph{g-vector} of $X$.
\end{definition}

\begin{theorem}
\label{gvector}
\cite[Theorem 5.5]{AIR}
The map $T\mapsto g^T$ gives an injection from the set of isomorphism classes of two-term presilting complexes
to $K_0(\proj \Lambda)$.
\end{theorem}


\section{Split-by-nilpotent extensions and two-term silting complexes}
Split-by-nilpotent extensions were first studied in \cite{AZ}.
We recall the definition of split-by-nilpotent extensions
and give a reduction theorem for 2-term silting complexes.

\begin{definition} Let $\Lambda$ and $A$ be two finite dimensional algebras. We call $\Lambda$ a {\bf split-by-nilpotent 
extension} of $A$ if there is a split surjective algebra homomorphism $\Lambda\to A$ whose kernel $L$
is nilpotent. 
\end{definition}

A path $a_0\to a_1\to\cdots\to a_n$ of $Q$ with $a_i\neq a_j$ for any pair $(i, j)\neq (0,n)$ ($i<j$) is said to be an \emph{$n$-cycle} provided $a_0=a_n$.
A 1-cycle is nothing but a loop.

Put $Q^\circ$ as the quiver obtained from $Q$ by 
removing all loops.
Note that if $Q$ has no $n$-cycle for $n>1$, then $Q^\circ$ is acyclic, whence $kQ^\circ$ is finite dimensional and hereditary.

We give an example of split-by-nilpotent extensions constructed  from quivers.

\begin{proposition}
\label{sbne} Let $Q$ be a finite quiver without $n$-cycles for $n>1$ and $L$ the two-sided ideal of $kQ$ generated by all loops.
Let $I$ be an admissible ideal of $kQ$ contained in $L$.
Then $\Lambda:=kQ/I$ is a split-by-nilpotent extension of $A:=kQ^\circ$.
\end{proposition}

\begin{proof}
Let $\iota:kQ^\circ\to kQ$ and $\pi:kQ\to kQ/L$ be the canonical algebra homomorphisms.
Since $Q^\circ$ is given by removing all loops, we see that the composition $\pi\circ \iota$ is an isomorphism.
Regarding $\pi$ as an algebra homomorphism to $A$,
it follows from $I\subseteq L$ that it gives rise to an algebra epimorphism $\widetilde{\pi}:\Lambda\to A$.

Put $\widetilde{\iota}:A\to \Lambda$ as the algebra homomorphism induced by $\iota$, and then we can easily check that $\widetilde{\pi}\circ \widetilde{\iota}=\mathrm{id}_{A}$. 
Thus, it is obtained that $\widetilde{\pi}$ is a split epimorphism with kernel $L/I$. 
Since $I$ is an admissible ideal and $L$ is contained in $\Rad kQ$,  $L/I$ is clearly nilpotent.
\end{proof}

In the rest of this section, let $\Lambda$ be a split-by-nilpotent extension of a finite dimensional algebra $A$.

Then, we have two functors: One is a faithful and dense functor $-\otimes_{A} \Lambda:\proj A\to \proj \Lambda.$
The other is a full and dense functor $-\otimes_{\Lambda} A: \proj \Lambda\to \proj A$.
Note that there is a natural isomorphism between $-\otimes_A\Lambda\otimes_\Lambda A$ and the identity functor of $\proj A$.

We show the following easy facts.

\begin{lemma}
\label{split}
Let $f:X\to Y$ be a $\Lambda$-homomorpshim between projective $\Lambda$-modules $X$ and $Y$.
If $f\otimes_\Lambda A$ is a split epimorphism, then so is $f$.
\end{lemma}
\begin{proof}
By $\Lambda\simeq \Lambda\otimes_\Lambda A\otimes_A\Lambda\ (\lambda\leftrightarrow 1\otimes_\Lambda 1\otimes_A \lambda)$ of right $\Lambda$-modules,
we have isomorphisms
$X\simeq X\otimes_\Lambda A\otimes_A \Lambda$ and $Y\simeq Y\otimes_\Lambda A\otimes_A \Lambda$.
Taking these into account, we can regard $f':=f\otimes_\Lambda A\otimes_A\Lambda$ as a $\Lambda$-homomorphism from $X$ to $Y$.
Then, it is easy to check that $(f-f')\otimes_\Lambda A$ is zero, whence $f-f'$ belongs to the radical of $\Hom_{\Lambda}(X,Y)$.
Since $f\otimes_\Lambda A$ is a split epimorphism, so is $f'=f\otimes_\Lambda A\otimes_A\Lambda$.
Hence, we conclude that $f$ is also a split epimorphism.
\end{proof}

\begin{proposition}
\label{indecomposability}
The triangle functor $-\otimes_{A} \Lambda:\Kb(\proj A)\to \Kb(\proj \Lambda) $ preserves the indecomposability of objects.
\end{proposition}
\begin{proof}
Let $X$ be an indecomposable object of $\Kb(\proj A)$. If $X\otimes_{A}\Lambda$
 is the direct sum of $Y_1$ and $Y_2$, then we obtain isomorphisms
 \[X\simeq X\otimes_{A} \Lambda \otimes_{\Lambda} A\simeq Y_1\otimes_{\Lambda}A\oplus Y_2\otimes_{\Lambda} A.\]
 Hence, one has $Y_1\otimes_{\Lambda} A=0$ or $Y_2\otimes_{\Lambda} A=0$.
 
 Thus, we have to show that $Y\otimes_\Lambda A=0$ implies $Y=0$.
 Each object $Y$ of $\Kb(\proj \Lambda)$ has a form 
 \[0\to Y^{a}\to\cdots \to Y^{b-1}\xrightarrow{d} Y^{b}\to 0.\]
 If $Y\otimes_{\Lambda} A=0$, then the complex 
 \[0\to Y^{a}\otimes_{\Lambda} A\to\cdots \to Y^{b-1}\otimes_{\Lambda}A \xrightarrow{d\otimes_\Lambda A} Y^{b}\otimes_{\Lambda} A\to 0\]
 splits out. In particular, $d\otimes_{\Lambda} A$ is a split epimorphism, which implies that
 $d$ is also a split epimorphism by Lemma \ref{split}.
 Repeating this argument leads to the conclusion that
 \[0\to Y^{a}\to\cdots \to Y^{b-1}\stackrel{d}{\to} Y^{b}\to 0\]
also splits out, and hence $Y$ must be $0$.
\end{proof}

Put $\widehat{-}:=-\otimes_{A} \Lambda$.

Now, we consider split-by-nilpotent extensions satisfying the following condition.

\begin{condition}\label{CG}
\begin{enumerate}[(a)]
%
\item For any indecomposable projective $A$-modules $P$ and $P'$, the subspace $\widehat{\Hom_A(P,P')}$ generates $\Hom_\Lambda(\widehat{P}, \widehat{P'})$ as a right $\End_\Lambda(\widehat{P})$-module.

\item Let $P:=\bigoplus_{i=0}^rP_i$ and $P':=\bigoplus_{j=1}^sP'_j$ be indecomposable decompositions of projective $A$-modules $P$ and $P'$. 
Let $f={}^t(f_{i,j})$ be an $A$-homomorphism from $P$ to $P'$, where $f_{i,j}:P_i\to P'_j$ and ${}^tG$ stands for the transpose of a matrix $G$.
If a $\Lambda$-endomorphism $l_0$ of $\widehat{P_0}$ is given, then there exist $\Lambda$-endomorphisms $l_i$ and $l'_j$ of $\widehat{P_i}$ and $\widehat{P'_j}$ satisfying $\widehat{f_{i,j}}\circ l_i=l'_j\circ \widehat{f_{i,j}}$. 
%
\end{enumerate}
\end{condition}


We state a main theorem of this section.

\begin{theorem}
\label{mainresult}
Assume that Condition \ref{CG} holds.
Then the following hold:
\begin{enumerate}[(1)]
\item Let $T$ and $T'$ be two-term presilting complexes of $\Kb(\proj A)$.
Then we have implications:
\[\Hom_{\Kb(\proj A)}(T,T'[1])=0 \iff \Hom_{\Kb(\proj \Lambda)}(\widehat{T},\widehat{T'}[1])=0.\]
In particular, every indecomposable two-term presilting complex of $\Kb(\proj A)$ is sent to that of $\Kb(\proj\Lambda)$ by $\widehat{-}$.
\item The functor $-\otimes_{A} \Lambda:\Kb(\proj A)\to \Kb(\proj \Lambda)$ induces a poset
isomorphism $\tsilt A\to \tsilt \Lambda$.
\end{enumerate}
\end{theorem}


\begin{proof}
(1) Let $T=[T_1\xrightarrow{d_T}T_0]$ and $T'=[T'_1\xrightarrow{d_{T'}}T'_0]$.
We may suppose that $T$ and $T'$ are indecomposable.
Note that $\add T_0\cap \add T_1=\{0\}=\add T'_0\cap \add T'_1$. 

Assume $\Hom_{\Kb(\proj A)}(T,T'[1])=0$.
Decomposing $T_1$ and $T'_0$ into $T_1=\bigoplus_{p=1}^rP_p$
and $T'_0=\bigoplus_{q=1}^sP'_q$, we observe that $P_p$ and $P'_q$ are not isomorphic by \cite[Lemma 2.25]{AI}.
Denote the canonical epimorphisms and inclusions by
\[\begin{array}{c@{\hspace{1cm}\mbox{and}\hspace{1cm}}c}
\xymatrix{
\widehat{T_1} \ar@<2pt>[r]^{\pi_p} & \widehat{P_p} \ar@<2pt>[l]^{\iota_p}
} & 
\xymatrix{
\widehat{T'_0} \ar@<2pt>[r]^{\pi'_q} & \widehat{P'_q} \ar@<2pt>[l]^{\iota'_q}
}
\end{array}.\]
From the condition (a), 
we have only to show the following claim:
\begin{claim*}
Let $1\leq p\leq r$ and $1\leq q\leq s$.
For any $A$-homomprhism $f:P_p\to P'_q$ and $\Lambda$-endomorphism $l_p$ of $\widehat{P_p}$, the composition $\iota'_q\circ\widehat{f}\circ l_p\circ \pi_p$ is zero in $\Kb(\proj\Lambda)$.
That is, there exist 
$g:\widehat{T_1}\to \widehat{T'_1}$ and
 $g':\widehat{T_0}\to \widehat{T'_0}$ such that
\[\iota'_q\circ \widehat{f}\circ l_p \circ \pi_p=\widehat{d_{T'}}\circ g+g'\circ \widehat{d_T}.\]
\end{claim*}

By $\Hom_{\Kb(\proj A)}(T,T'[1])=0$, 
we have $h:T_1\to T'_1$ and $h':T_0\to T'_0$ such that
\begin{equation}\label{ht0}
\iota'_{q}\circ \widehat{f} \circ \pi_{p}=\widehat{d_{T'}}\circ \widehat{h}+\widehat{h'}\circ \widehat{d_{T}}.
\end{equation}
%
%
Then, one obtains equalities
\[\def\arraystretch{1.8}
\begin{array}{rcl}
\iota'_q\circ \widehat{f}\circ l_{p} \circ \pi_{p}
&=& \iota'_{q}\circ \widehat{f}\circ \pi_{p}\circ \iota_{p}\circ l_{p} \circ \pi_{p}\\
&=& (\widehat{d_{T'}}\circ \widehat{h}+\widehat{h'}\circ \widehat{d_{T}})\circ \iota_{p}\circ l_{p} \circ \pi_{p}\\
&=& \widehat{d_{T'}}\circ \widehat{h} \circ  \iota_{p}\circ l_{p} \circ \pi_{p}
+ \widehat{h'}\circ \widehat{d_{T}} \circ \iota_{p}\circ l_{p} \circ \pi_{p}\\
&=& \widehat{d_{T'}}\circ \widehat{h} \circ  \iota_{p}\circ l_{p} \circ \pi_{p}+
{\displaystyle\sum_{t=1}^s} \iota'_t\circ \pi'_t\circ \widehat{h'}\circ \widehat{d_{T}}\circ \iota_{p}\circ l_{p} \circ \pi_{p}.
\end{array}                                          
\]
Here, we have used that $\pi_p\circ \iota_p$ and $\sum_{t=1}^s\iota'_t\circ \pi'_t$ are identities at the first and the last equalities.

Consider the homomorpshim $(f_{u,t}):=(\pi'_t\circ h'\circ d_T\circ \iota_u): \bigoplus_{u=1}^rP_u\to\bigoplus_{t=1}^sP'_t$.
Since we have an endomorphism $l_p$ of $\widehat{P_p}$, applying the condition (b) one gets endomorphisms $l_u$ an $l'_t$ of $\widehat{P_u}$ and $\widehat{P'_t}$ satisfying $\widehat{f_{u,t}}\circ l_u=l'_t\circ \widehat{f_{u,t}}$.
From the equality above, we see
\[\iota'_{q}\circ \widehat{f}\circ l_{p} \circ \pi_{p}=\widehat{d_{T'}}\circ \widehat{h}  \circ  \iota_{p}\circ l_{p} \circ \pi_{p}+ \sum_{t=1}^s\iota'_{t}\circ l'_t\circ \widehat{f_{p,t}} \circ \pi_{p}.\]
Taking notice to the identity $\sum_{u=1}^r\iota_u\circ \pi_u$,
we find
\[\def\arraystretch{1.8}
\begin{array}{rcl}
{\displaystyle\sum_{t=1}^s} \iota'_{t}\circ l'_t\circ \widehat{f_{p,t}} \circ \pi_{p}
&=& {\displaystyle\sum_{t=1}^s} \iota'_{t}\circ l'_t\circ \pi'_t\circ\widehat{h'}\circ \widehat{d_T}-{\displaystyle\sum_{t=1}^s\sum_{u\neq p}} \iota'_{r}\circ l'_t\circ \widehat{f_{u,t}}\circ \pi_u \\
&=& {\displaystyle\sum_{t=1}^s} \iota'_{t}\circ l'_t\circ \pi'_t\circ\widehat{h'}\circ \widehat{d_T}-{\displaystyle\sum_{t=1}^s} {\displaystyle\sum_{u\neq p}} \iota'_t \circ\widehat{f_{u,t}}\circ l_u\circ \pi_u \\
&=& {\displaystyle\sum_{t=1}^s} \iota'_{t}\circ l'_t\circ \pi'_t\circ\widehat{h'}\circ \widehat{d_T} -{\displaystyle\sum_{u\neq p}} \widehat{h'}\circ \widehat{d_T}\circ \iota_u\circ l_u\circ \pi_u \\
&=& {\displaystyle\sum_{t=1}^s} \iota'_{t}\circ l'_t\circ \pi'_t\circ\widehat{h'}\circ \widehat{d_T}
+{\displaystyle\sum_{u\neq p}} \widehat{d_{T'}}\circ \widehat{h}\circ \iota_u\circ l_u\circ \pi_u.
\end{array}\]
Here, (\ref{ht0}) has been used at the last equality.
Therefore, we obtain 
\[\iota'_q\circ \widehat{f}\circ l_p\circ \pi_p=\widehat{d_{T'}}\circ \left({\displaystyle\sum_{u=1}^r} \widehat{h}\circ \iota_u\circ l_u\circ \pi_u \right)+\left({\displaystyle\sum_{t=1}^s} \iota'_t\circ l'_t\circ \pi'_t\circ \widehat{h'}\right)\circ \widehat{d_T},\]
which is the desired equality.

Next, we assume $\Hom_{\Kb(\proj \Lambda)}(\widehat{T},\widehat{T'}[1])=0.$
It is then easy to check that $\Hom_{\Kb(\proj A)}(T,T'[1])=0$.
In fact, the composition $(-\otimes_{\Lambda}A)\circ(-\otimes_{A}\Lambda)$ is the identity functor.

Now, the assertion (1) follows from Proposition \ref{indecomposability}
 
 (2) Let $T$ be a basic two-term silting complex in $\Kb(\proj A)$. 
 Since $A$ belongs to $\thick T$, we obtain that $\Lambda$ is in $\thick\widehat{T}$.
Then, Proposition \ref{indecomposability} and (1) imply 
that the functor $-\otimes_{A} \Lambda:\Kb(\proj A)\to \Kb(\proj \Lambda)$ induces a poset inclusion $\tsilt A\to \tsilt\Lambda$.

We show the surjectivity.
Let $U$ be a basic two-term silting complex in $\Kb(\proj \Lambda)$.
Since the functor $-\otimes_{\Lambda} A: \proj \Lambda\to \proj A$ is full, it is seen that
$U\otimes_{\Lambda} A$ is a two-term silting complex in $\Kb(\proj A)$.
Hence, we get a two-term silting complex $U\otimes_{\Lambda} A\otimes_A \Lambda$.
Observe that the g-vector of $U\otimes_{\Lambda} A\otimes_A \Lambda$ coinsides with that of $U$, whence they are isomorphic by Theorem \ref{gvector}.
\end{proof}
 
 
\section{Main results}

In this section, we realize our goal of this paper.

Let $i$ be a vertex of $Q$.
Recall that $e_i$ denotes the primitive idempotent of $\Lambda$.
Now, put $P_i=e_i\Lambda$ and $X_i=e_i\Lambda/e_i\Lambda(1-e_i)\Lambda$, which is isomorphic to $\Lambda/(1-e_i)$.
Here, for an element $x$ of $\Lambda$, $(x)$ stands for the two-sided ideal of $\Lambda$ generated by $x$.
We immediately obtain that $X_i$ is a support $\tau$-tilting module with $\supp(X_i)=\{i\}$ and an arrow $X_i\to0$ in $\H(\sttilt\Lambda)$.

We observe the following result.

\begin{lemma}\label{1to2}
Let $Q$ be a quiver with precisely two vertices, say  $1,2$,
and $I$  an admissible ideal of $kQ$. Suppose that there is an arrow $\alpha$ from $1$ to $2$.
Put $\Lambda:=kQ/I$.
Then $X_1\oplus P_2 $ is not $\tau$-rigid.
Moreover, we have
$\Hom_{\Lambda}(P_1,\tau X_1)=0$ if and only if 
$\alpha$ is a unique arrow from $1$ to $2$ and 
$\alpha\Lambda e_2=e_1\Lambda\alpha$.  
\end{lemma}
\begin{proof} 
%
Since a minimal projective presentation of $X_1$ has the form
\[P_2^{\oplus r}\to P_1 \to X_1\to 0\]
with $r>0$, there exists a non-zero homomorphism from $P_2$ to $\tau X_1$. 
Hence, we see that $X_1\oplus P_2$ is not $\tau$-rigid.

Assume $\Hom_{\Lambda}(P_1,\tau X_1)=0$.
It is observed that $\tau X_1$ is a support $\tau^-$-tilting module with support $\{2\}$, which implies $\tau X_1=D(\Lambda e_2/\Lambda e_1\Lambda e_2)$.
Therefore, we get a minimal projective presentation
\[e_2\Lambda\stackrel{f}{\to} e_1\Lambda\to X_1\to 0\] 
of $X_1$, whence there is an epimorphism $f:e_2\Lambda\to e_1\Lambda e_2\Lambda$.
This leads to the conclusion that $e_1\Lambda e_2\Lambda$ has a simple top, and so there is no arrow from $1$ to $2$ other than $\alpha$.

We show $\alpha\Lambda e_2=e_1\Lambda\alpha$.
Without loss of generality, we may assume that $f$ is the left multiplication by $\alpha$.
Then, one has an inclusion
\[e_1\Lambda \alpha\subseteq e_1\Lambda e_2 \Lambda=\alpha \Lambda,\]
which yields $e_1\Lambda \alpha=e_1\Lambda \alpha e_2\subseteq\alpha \Lambda e_2$.
On the other hand, we have the following commutative diagram:
\[\xymatrix{
(e_1\Lambda)^* \ar[d] \ar[r]^{f^*} & (e_2\Lambda)^* \ar[d] \ar[r] & \Tr X_1 \ar[r] & 0 \\
\Lambda e_1 \ar[r]_{-\cdot\alpha} & \Lambda e_2 \ar[r] & \Lambda e_2/\Lambda \alpha \ar[r] & 0
}\]
of exact sequences,
where $(-)^*$ and $\Tr$ stand for the $\Lambda$-dual and the transpose. 
Here, the vertical arrows are isomorphisms.
We get isomorphisms
\[\Lambda e_2/\Lambda \alpha\simeq \Tr X_1 \simeq D\tau X_1 \simeq \Lambda e_2/\Lambda e_1 \Lambda e_2.\]
By $\Lambda \alpha\subseteq \Lambda e_1 \Lambda e_2$, one sees that $\Lambda \alpha= \Lambda e_1 \Lambda e_2$,
whence
\[\alpha \Lambda e_2\subseteq \Lambda e_1 \Lambda e_2=\Lambda \alpha.\]
Thus, it is obtained that $\alpha \Lambda e_2=e_1\alpha \Lambda e_2\subseteq e_1\Lambda \alpha.$

Next, assume that there is a unique arrow $\alpha$ from $1$ to $2$ and $\alpha\Lambda e_2=e_1\Lambda\alpha$.
We have 
\begin{equation}\label{rightideal}
e_1 \Lambda e_2 \Lambda= e_1\Lambda \alpha \Lambda =\alpha \Lambda e_2\Lambda =\alpha \Lambda,
\end{equation}
which implies that the sequence
\[e_2\Lambda\xrightarrow{\alpha\cdot -} e_1\Lambda\to X_1 \to 0\] 
is a minimal projective presentation of $X_1$.
Hence, the transpose of $X_1$ is isomprhic to $\Lambda e_2/\Lambda \alpha$.
The dual argument of (\ref{rightideal}) leads to the conclusion
that $\Lambda \alpha=\Lambda e_1 \Lambda e_2$,
whence $\tau X_1 \simeq D(\Lambda e_2/\Lambda e_1 \Lambda e_2)$.
Thus, it finds out that there is no non-zero homomorphism $P_1\to \tau X_1$.
\end{proof}


The following proposition is an immediate consequence of Lemma \ref{1to2}.

\begin{proposition}\label{mr}
Let $Q, \alpha$ and $\Lambda$ be as in Lemma \ref{1to2}.
\begin{enumerate}[(1)]
\item $\H(\sttilt\Lambda)$ is of type $\widetilde{A}_{3,2}$ if and only if the following hold.
\begin{enumerate}[(a)]
\item $\alpha$ is a unique arrow of $Q$ from $1$ to $2$ and $\alpha\Lambda e_2=e_1\Lambda \alpha$.
\item There is no arrow of $Q$ from $2$ to $1$. 
\end{enumerate} 

\item $\H(\sttilt\Lambda)$ is of type $\widetilde{A}_{3,3}$ if and only if the following hold.
\begin{enumerate}[(a)]
\item $\alpha$ is a unique arrow of $Q$ from $1$ to $2$ and $\alpha\Lambda e_2=e_1\Lambda \alpha$.
\item There exists a unique arrow $\beta$ of $Q$ from $2$ to $1$ and $\beta\Lambda e_1=e_2\Lambda \beta$.
\end{enumerate} 
\end{enumerate}
\end{proposition}
\begin{proof}
We only show the statement (1): The other can be handled by a similar way.

Assume that $\H(\sttilt\Lambda)$ is of type $\widetilde{A}_{3,2}$, say 
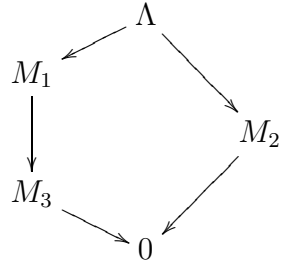
\[\xymatrix@R=0.2cm{
& \Lambda \ar[dl] \ar[ddr] & \\
M_1 \ar[dd] & &  \\
 && M_2 \ar[ddl] \\
M_3 \ar[dr] && \\
& 0 & 
}\]
Since $X_1$ is a support $\tau$-tilting module having an arrow $X_1\to0$ but not projective, $M_3$ must be $X_1$.
Therefore, $M_2$ finds out to be $X_2$, which is a projective module $P_2$.
This implies that there is no non-zero homomorphism $P_1\to P_2$, whence $Q$ has no arrow $2\to 1$.
It follows from Theorem \ref{basicprop} that $X_1$ and $P_1$ are direct summands of $M_1$, that is, $M_1=X_1\oplus P_1$,
whence $\Hom_\Lambda(P_1, \tau X_1)=0$.
As Lemma \ref{1to2}, we have the assertion (a).


The converse can be checked easily.
\end{proof}


Let $M$ be a support $\tau$-tilting module.
We denote by $\dpr(M)$ the set of direct predecessors of $M$ in $\H(\sttilt\Lambda)$.
Note that for any module $N$ in $\dpr(X_i)$, the number of vertices in $\supp(N)$ is precisely 2.

In the rest of this section, let $\ORA{\T}$ be a tree quiver.

The following result plays an important role.

\begin{lemma}\label{determiningarrow} 
Assume that $\sttilt\Lambda\simeq \sttilt k\ORA{\T}$. 
\begin{enumerate}[(1)]
\item 
There is a unique support $\tau$-tilting module $M_{i,j}$ in $\dpr(X_i)\cup \dpr(X_j)$ satisfying $M_{i,j}\geq X_i ,X_j $. 
\item The interval $[0,M_{i,j}]$ coincides with $\sttilt_P\Lambda$,
where $P:=\bigoplus_{\ell\neq i,j}P_\ell$,
and it has one of the following forms:
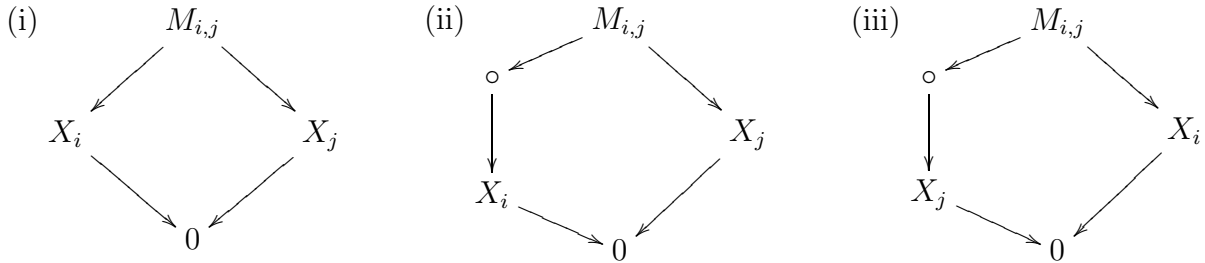
\[\begin{array}{@{{\rm (i)}}c@{\hspace{1cm}{\rm (ii)}}c@{\hspace{1cm}{\rm (iii)}}c}
\xymatrix{
& M_{i,j} \ar[dl] \ar[dr] & \\
X_i \ar[dr] & & X_j \ar[dl] \\
& 0 & 
} &
\xymatrix@R=0.2cm{
& M_{i,j} \ar[dl] \ar[ddr] & \\
\circ \ar[dd] & &  \\
 && X_j \ar[ddl] \\
X_i \ar[dr] && \\
& 0 & 
} &
\xymatrix@R=0.2cm{
& M_{i,j} \ar[dl] \ar[ddr] & \\
\circ \ar[dd] & &  \\
 && X_i \ar[ddl] \\
X_j \ar[dr] && \\
& 0 & 
}\end{array}\]
\item In the case of $\mathrm{(i)}$, there is no arrow of $Q$ between $i$ and $j$.
\item In the case of $\mathrm{(ii)}$, there is a unique arrow of $Q$ from $i$ to $j$.
\end{enumerate} 
\end{lemma}
\begin{proof}
Put $e=1-e_i-e_j$, $P=e \Lambda$ and note that $\sttilt_P\Lambda\simeq \sttilt\Lambda/(e)$.
We also remark that for every $M\in\dpr(X_i)\cup \dpr(X_j)$ with $M\geq X_i, X_j$, the support of $M$ is just $\{i,j\}$.

We first consider the case $\Lambda=k\ORA{\T}$.
Since $\T$ is a tree, we observe that one of the following isomorphisms hold:
\[\Lambda/(e)\simeq\begin{cases}
\ X_j\oplus e_i\Lambda/e_i\Lambda e\Lambda & \mbox{if $Q$ has an arrow $i\to j$}; \\
\ X_i\oplus e_j\Lambda/ e_j\Lambda e \Lambda & \mbox{if $Q$ has an arrow $j\to i$}; \\
\ X_i\oplus X_j & \mbox{if $Q$ has no arrow between  $i$ and $j$}.
\end{cases}\]
Note that $\Lambda/(e)$ is a maximum element of $\sttilt_P \Lambda$.
%
%
Thus, all the assetions hold.

Let $\Lambda$ be an arbitrary algebra with $\sttilt\Lambda\simeq\sttilt k\ORA{\T}$.
It is evident to satisfy the assertion (1).
Let $\rho:\sttilt\Lambda\xrightarrow{\sim} \sttilt k\ORA{\T}$.
By the argument for $\Lambda=k\ORA{\T}$, we see that
$[0,M_{i,j}]\simeq[0,\rho(M_{i,j})]$ has one of the forms (i), (ii) and (iii).

We show $\sttilt_P\Lambda=[0, M_{i,j}]$ to establish the assertion (2).
As $\supp(M_{i,j})=\{i,j\}$, one sees that $[0,M_{i,j}]$ is an induced subposet of $\sttilt_P\Lambda$.
It follows from $\sttilt_P\Lambda\simeq \sttilt\Lambda/(e)$
that $\sttilt_P\Lambda$ is a 2-regular poset.
Since $[0,M_{i,j}]$ is also 2-regular, it coincides with $\sttilt_P\Lambda$.

We show the assertion (3). The form (i) implies that $X_i$ and $X_j$ are projectives of $\Lambda/(e)$.
From $\supp(X_i)=\{i\}$ and $\supp(X_j)=\{j\}$, there is no arrow between $i$ and $j$.

Finally, we show the asssertion (4). Since $X_j$ is a projective $\Lambda/(e)$-module, there is  no arrow from $j$ to $i$.
If the quiver of $\Lambda/(e)$ is not connected, then $\sttilt_P \Lambda\simeq \sttilt \Lambda/(e)$ is of the form (i).
So, it is connected, whence there is an arrow $\alpha$ from $i$ to $j$. Such an arrow is unique by Proposition \ref{mr}.
%
%
\end{proof}

\begin{remark}
\label{remark}
Lemma \ref{determiningarrow} also shows that if $\sttilt\Lambda\simeq \sttilt k\ORA{\T}$, then $Q^\circ\simeq \ORA{\T}$.
Actually, a poset isomorphism $\rho:\sttilt\Lambda\xrightarrow{\sim} \sttilt k\ORA{\T}$ 
induces a bijection $\sigma=\sigma_\rho :Q_0\to \ORA{\T}_0$ 
by the following correspondence:
\[\rho(X_i)=S_{\sigma(i)},\]
where $S_i$ denotes a simple module corresponding to a vertex $i$. 
Then by Lemma \ref{determiningarrow}, $\sigma$ can be extended to a quiver automorphism
\[\sigma:Q^\circ\xrightarrow{\sim}\ORA{\T}.\]           
\end{remark}

Let $\ORA{A_n}:=1\xrightarrow{\alpha_1} 2\xrightarrow{\alpha_2}\cdots\xrightarrow{\alpha_{n-1}} n$.

Now, assume that $\sttilt\Lambda\stackrel{\rho}{\simeq} \sttilt k\ORA{\T}$,
and by a suitable reindexing, we suppose that $\rho(X_i)=S_i$ and $Q^\circ=\ORA{\T}$.
Let $i_1\to\cdots\to i_r$ be a path of $\ORA{\T}$.
For $t=0,1,\cdots, r-1$,
put 
\[\def\arraystretch{1.5}
\begin{array}{rl}
e(t) &:= 1-e_{i_{r}}-\cdots-e_{i_{r-t}} \\
J(t) &:=J_\Lambda(t):=\Lambda/(e(t)).
\end{array}\]
Note that $J(t)$ is the maximum element of $\sttilt_{e(t)\Lambda}\Lambda$ with support $\{i_r,\cdots, i_{r-t}\}$ and that $J(t)=J(t-1)\oplus e_{i_{r-t}} \Lambda/e_{i_{r-t}} \Lambda e(t) \Lambda$.
\begin{lemma}\label{determiningjoin}
Under the setting above, the following hold.
\begin{enumerate}[(1)]
\item Let $M$ be a support $\tau$-tilting module of $\Lambda$ and fix $t=1,\cdots,r-1$.
 Then $M=J(t)$ if and only if it satisfies the following:
\begin{enumerate}[(a)]
\item $M \geq X_{i_{r-t}}$;
\item There exists an arrow $M \to J(t-1)$ of $\H(\sttilt\Lambda)$.
\end{enumerate}

\item $J(r-1)=\bigvee_{1\leq \ell\leq r}X_{i_\ell}$ if $\sttilt\Lambda$ is a lattice.

\item We have a poset isomorphism 
\[\sttilt \Lambda/(e(r-1))\simeq \sttilt k\ORA{A_r}.\] 
\end{enumerate}
\end{lemma}
\begin{proof}
(1) Looking at the supports, we obtain that $X_{r-t}$ and $J(t-1)$ belong to $\sttilt_{e(t)\Lambda}\Lambda$,
 and so they are less than $J(t)$.
Letting $Q'$ be the quiver of $\Lambda/(e(t))$, it is observed that $(Q')^\circ$ is of the form $i_{r-t}\to\cdots\to i_r$.
Since $(J(t-1), e_{i_{r-t}}\Lambda/(e(t)))$ is a support $\tau$-tilting pair of $\Lambda/(e(t))$,
 we have an arrow $J(t)\to J(t-1)$ of $\H(\sttilt\Lambda)$.

Conversely, let $M$ be a direct predecessor of $J(t-1)$ in $\H(\sttilt\Lambda)$.
Assume that $M\geq X_{i_{r-t}}$.
We then see that the support of $M$ contains $\{i_r,\cdots,i_{r-t}\}$. 
Since $M$ is a direct predecessor of $J(t-1)$,
the gap between $\supp(M)$ and $\supp(J(t-1))$ is at most one.
By $\supp (J(t-1))=\{i_r,\dots,i_{r-t+1}\}=\{i_r,\dots,i_{r-t}\}\setminus \{i_{r-t}\}$, 
we obtain $\supp (M)=\{i_r,\cdots,i_{r-t}\}$.
Taking the support $\tau$-tilting pair $(J(t-1), P)$,
it is seen that the almost complete support $\tau$-tilting pair $(J(t-1), P/P_{i_{r-t}})$ is also a direct summand of the support $\tau$-tilting pair $(M, P/P_{i_{r-t}})$.
By Theorem \ref{basicprop} (2), it finds out that $M$ is just $J(t)$.
%
%

(2) Note that taking $\vee$ makes sense since $\sttilt\Lambda$ is a lattice.
For $1\leq t\leq r-1$, it follows from (1) that $X_{i_{r-t}}\leq J(t)\to J(t-1)$. Therefore, we have $J(t)\geq J(t-1)\vee X_{i_{r-t}}\geq J(t-1)$, which implies that $J(t)=J(t-1)\vee X_{i_{r-t}}$. 
Consequently, one sees
\[J(r-1)=J(0)\vee X_{i_{r-1}}\vee \cdots\vee X_{i_1},\]
whence $J(r-1)=\bigvee_{1\leq \ell\leq r}X_{i_\ell}$ by $J(0)=X_{i_r}$.

(3)  
By $\sttilt_{e(t)\Lambda}(\Lambda)\simeq \sttilt\Lambda/(e(t))$, we get a poset isomorphism
\[\sttilt\Lambda/(e(t))\simeq [0,J(t)].\] 
As $J_\Lambda(0)=X_r$,
we observe $\rho(J_\Lambda(0))=J_{k\ORA{\T}}(0)$.
Applying (1) to $k\ORA{\T}$, one inductively obtains $\rho(J_\Lambda(t))=J_{k\ORA{\T}}(t)$.

Now, we complete the proof.
One has already seen a poset isomorpshim $[0, J_{k\ORA{\T}}(t)]\simeq \sttilt k\ORA{\T}/(e(t))$.
Since the algebra $k\ORA{\T}/(e(t))$ is actually $k[i_{r-t}\to i_{r-t+1}\to\cdots\to i_r]$, that is, $k\ORA{A_{r+1}}$,
we have 
\[\sttilt\Lambda/(e(t))\simeq [0,J_\Lambda(t)]\simeq [0,\rho(J_\Lambda(t))]= [0, J_{k\ORA{\T}}(t)]\simeq \sttilt k\ORA{A_{t+1}}.\]
In particular, one gets the desired poset isomorphism.
\end{proof}

To realize our goal, we first observe algebras of type $A_n$.

\begin{lemma}\label{teqlem}
 Let $\Lambda=k\ORA{A_n}$.
Put $T_1=\Lambda/P_1$ and $T_i:=(\Lambda/P_i )\oplus S_{i-1} (i\neq 1)$. Note that ${T_i}'s$ are support $\tau$-tilting modules.
Then we have the following:
\begin{enumerate}[(1)]

\item $M:=\bigwedge_{i\neq 1} T_{i}$ is a tilting module and $J:=\bigvee_{i\neq n}S_i$ is isomorphic to $\Lambda/(e_n)$.
In particular, $M\not\leq J$. 

\item There exists a unique path from $\Lambda$ to $0$ in $\H(\sttilt\Lambda)$ with length $n$.
Moreover, it factors through $T_1$.
\end{enumerate}
\end{lemma} 

\begin{proof}
(1) Let $M'$ be a minimum element of $\sttilt^{P_1}\Lambda$.
Observe $T_i\geq M'$ for all $i\neq 1$.
Therefore, we  have $M\geq M'$, which implies that $\supp(M)=\{1,2,\dots, n\}$ since $P_1$ is a direct summand of $M'$.
Hence, it figures out that $M$ is a tilting module by \cite[Proposition 2.2]{AIR}.
 
We show that $J$ is isomorphic to $\Lambda/(e_n)$.
Note that the quiver of $\Lambda/(e_n)$ is 
\[1\to 2\to\cdots\to n-1.\]
Applying Lemma \ref{determiningjoin} to this path of $\ORA{A_n}$,
one has $J(n-2)=\bigvee_{1\leq i\leq n-1}S_i$,
whence $\Lambda/(e_n)=J(n-2)=J$.

The last assertion follows from $\supp(J)=Q_0\setminus \{n\}$.

We can check the assertion (2), directly. 
\end{proof}

Denote by $\trigid \Lambda$ the set of isomorphism classes of indecomposable $\tau$-rigid $\Lambda$-modules.

In the case of $Q^\circ=\ORA{A_n}$, we write $X(i, n):=P_i$ for any $i$, and put $X(i,j):=e_i\Lambda/e_i\Lambda e_{j+1}\Lambda$ for any $i\leq j<n$.

We show the main theorem in the case of type $A_n$.

\begin{theorem}
\label{t1}
Let $Q$ be a quiver with $Q^{\circ}=\ORA{A_n}$ and $I$ an admissible ideal of $kQ$. Put $\Lambda=kQ/I$. Assume that
the composition $\alpha_1\cdots\alpha_{n-1}\neq0$ in $\Lambda$ and $e_{i}\Lambda \alpha_{i}=\alpha_{i}\Lambda e_{i+1} $  for any $i$.
\begin{enumerate}[(1)]
\item 
$\Hom_{\Lambda} (X(i,j),\tau X(p,q))=0$
if and only if one of the following holds:
\begin{enumerate}[(a)]
\item $[i-1, j]\cap[p-1, q]=\emptyset$;
\item $[i-1, j]\subseteq [p-1, q] $ or $[p-1, q]\subseteq [i-1, j]$;
\item $i-1<p-1\leq j<q$.
\end{enumerate}

\item We have $\trigid \Lambda=\{X(i,j)\mid 1\leq i\leq j\leq n\}$. Moreover, one obtains
\[\sttilt\Lambda\simeq \sttilt k\ORA{A_n}.\] 
\end{enumerate}
\end{theorem}
\begin{proof}
We show the assertion $(1)$. 
Suppose that $\Hom_{\Lambda} (X(i,j),\tau X(p,q))=0$ and neither (a) nor (b) holds.
Then, we should check that (c) holds.
Since we do not have either (a) or (b), one gets 
$p-1< i-1< q< j$ or $i-1< p-1< j< q$.
As $Q^\circ =\ORA{A_n}$, it is seen that $X(i,j)$ and $X(p,q)$ have $S_\ell$ as a composition fuctor, where $\ell$ runs from $i$ to $q$ or from $p$ to $j$.
Putting $e:=1-e_i-\cdots-e_q$ or $e:=1-e_p-\cdots-e_j$,
both $X(i,j)$ and $X(p,q)$ can be regarded as $\Lambda/(e)$-modules, from where we may assume that $[i,j]\cup[p,q]=\{1,2,\dots,n\}$ by Proposition \ref{basicfact} (2).
Thus, one has $j=n$ or $q=n$.
If $j=n$, then $q<n$ and $X(i,j)=P_i$.
By a direct calculation, we see 
\[\mathrm{Tr} X(p,q)\simeq \Lambda e_{q+1}/\Lambda e_p\Lambda e_{q+1},\]
which implies that $\supp(\tau X(p,q))=[p+1,q+1]$.
Since we are considering the case of $p-1<i-1\leq q$,
one observes that $i$ belongs to the support of $\tau X(p,q)$.
This makes a contradiction by $\Hom_\Lambda(X(i,j), \tau X(p,q))=0$, whence $q=n$, that is, $i-1< p-1\leq j< q$.

We show the converse.

The condition (a) says that $\supp(X(i,j))\cap\supp(\tau X(p,q))=\emptyset$, whence we are done.

Assume that the condition (b) holds. 
As the same argument above, we may suppose that $[i-1, j]=[0,n]$ or $[p-1,q]=[0,n]$.
If $q=n$, then $X(p,q)=P_p$, in which we have nothing to do.
Let $[i-1,j]=[0,n]$.
We get $X(i,j)=P_1$, and have already seen that $\supp(\tau X(p,q))=[p+1, q+1]$, whence $\Hom_\Lambda(X(i,j), \tau X(p,q))=0$.

We assume that the condition (c) holds.
Similarly, one can suppose that $q=n$. Then $X(p,q)=P_p$.
Thus, we have $\Hom_{\Lambda}(X(i,j),\tau X(p,q))=0.$

We show the assertion $(2)$.
From (1), observe that $X(i,j)$ is an indecomosable $\tau$-rigid module.
Also, note that if $\Lambda=k\ORA{A_n}$, then $X(i,j)$'s are all of the indecomposable $\Lambda$-modules.
Put $X:=\bigoplus_{1\leq i\leq j\leq n} X(i,j)$. 
Since the $\tau$-rigidity does not depend on the choice of $\Lambda$ by (1),
the attachment $X(i,j)\mapsto X(i,j)$ induces a poset isomorphism $\sttilt\Lambda\cap \add X\simeq\sttilt k\ORA{A_n}$,
from which it follows that $\sttilt\Lambda\cap \add X$ is a induced subposet of $\sttilt\Lambda$ and $n$-regular.
Hence, it figures out that $\sttilt\Lambda\cap \add X$ contains the connected component $\C_\Lambda$ of $\H(\sttilt\Lambda)$ having $\Lambda$.
The finiteness of $\sttilt\Lambda\cap \add X$ leads to the conclusion that $\C_\Lambda$ is a finite connected component, whence the equalites $\C_\Lambda=\sttilt\Lambda\cap \add X=\sttilt\Lambda$ hold by Theorem \ref{basicprop}.
Thus, the proof is completed.
%
\end{proof}

We understand the case of type $A_n$.

\begin{corollary}
\label{casea}
 Let $Q$ be a quiver with $Q^{\circ}=\ORA{A_n}$ and $I$ an admissible ideal of $kQ$.
Put $\Lambda=kQ/I$.
Then $\sttilt\Lambda\simeq \sttilt k\ORA{A_n}$ if and only if
$\alpha_1\cdots\alpha_{n-1}\neq0$ in $\Lambda$ and $e_{i}\Lambda \alpha_{i}=\alpha_{i}\Lambda e_{i+1} $  for any $i$.
\end{corollary}
\begin{proof}
We have only to show ``only if'' part by Theorem \ref{t1}.
Assume that $\sttilt\Lambda\simeq \sttilt k\ORA{A_n}$, and apply induction on $n$.
It is trivial that the assertion holds for $n=2$.

Let $n>2$ and $\Lambda_{i}:=\Lambda/(1-e_{i}-e_{i+1})$.
Applying Lemma \ref{determiningarrow} to $i\rightsquigarrow i,\ j\rightsquigarrow i+1$,
we see that $\H(\sttilt_P\Lambda)$ is of type $\widetilde{A}_{3,2}$,
where $P:=e_i\Lambda\oplus e_{i+1}\Lambda$,
whence $\H(\sttilt\Lambda_i)$ has the same type.
Therefore, we have 
\[e_{i}\Lambda \alpha_{i}=e_{i}\Lambda_{i} \alpha_{i}\stackrel{{\rm  \ref{mr}}}{=}\alpha_{i}\Lambda_{i} e_{i+1}= \alpha_{i}\Lambda e_{i+1}.\]

Let $J_{\Lambda}(p,q):=\bigvee_{p \leq k\leq q} X_k $. 
As Lemma \ref{determiningarrow}, the position of $X_k$ is uniquely determined by the poset structure of $\sttilt\Lambda$,
and hense, so is also $J_\Lambda(p,q)$.
Thus, we have
\[[0,J_{\Lambda}(i,j)]\simeq[0,J_{k\ORA{A_n}}(i,j)].\]
Applying Lemma \ref{determiningjoin} to a path $i\to\cdots\to j$, we get 
\[J_{\Lambda}(i,j)=\Lambda/(1-e_{i}-\cdots-e_{j}).\]
Consequently, one sees 
\[\sttilt(\Lambda/ (1-e_{i} \cdots -e_{j}))=[0,J_{\Lambda}(i,j)]\simeq [0,J_{k\ORA{A_n}}(i,j)]\simeq \sttilt k\ORA{A_{\ell}},\]
where $\ell=j-i+1$. 
By the induction hypothesis,
we obtain $\alpha_1\cdots\alpha_{n-2}\neq0$ and $\alpha_2\cdots\alpha_{n-1}\neq0$ in $\Lambda$.

Let $X(i,j):=e_i\Lambda/e_i\Lambda e_{j+1}\Lambda$ and
regard it as a $\Lambda/(e_n)$-module if $j<n$.
By $e_i(\Lambda/(e_n))\alpha_i=\alpha_i(\Lambda/(e_n))e_{i+1}$ and $\alpha_1\cdots\alpha_{n-2}\neq0$, it follows from Theorem \ref{t1} that
$\{X(i,j)\ |\ 1\leq i\leq j<n \}$ is a complete set of representatives of isomorphism classes of indecomposable $\tau$-rigid $\Lambda/(e_n)$-modules.
Moreover, we also obtain that $\Hom_{\Lambda/(e_n)}(X(i,j), X(p,q))=0$ if and only if the pair $((i,j), (p,q))$ satisfies one of the conditions (a), (b) and (c) in Theorem \ref{t1}.
A similar argument works for $1<i\leq j\leq n$.
Putting $T_1:=\Lambda/P_1$ and $T_i:=\Lambda/P_i\oplus X_{i-1}$ for $i>1$, it is seen that $T_i$ is a support $\tau$-tilting $\Lambda$-module with an arrow $\Lambda\to T_i$ in $\H(\sttilt\Lambda)$ for any $i$.

In the rest of this proof, 
we assume $\alpha_1\cdots\alpha_{n-1}=0$
and show that this makes a contradiction.

Let $T(r):=P_1\oplus P_{r+1}\oplus\cdots\oplus P_n\oplus\bigoplus_{\ell=2}^rX(1,\ell-1)$.
We observe that $T(r)/P_1$ is a $\tau$-rigid $\Lambda/(e_r)$-module, and so it is also $\tau$-rigid over $\Lambda$.
As $\alpha_1\cdots\alpha_{n-1}=0$, one sees that $P_1\oplus \bigoplus_{\ell=2}^r X(1, \ell-1)$ is a $\tau$-rigid $\Lambda/(e_n)$-module, and also over $\Lambda$.
Thus, it finds out that $T(r)$ is in $\sttilt\Lambda$.
By Theorem \ref{basicprop}, we have a path 
\begin{equation}\label{path1}
T_2=T(2)\to T(3)\to\cdots\to T(n-1)
\end{equation}
in $\H(\sttilt\Lambda)$.

Let $M_{\Lambda}:=\bigwedge_{i\neq 1}T_i$ and $J_{\Lambda}:=J_{\Lambda}(1,n-1)=\bigvee_{1\leq i\leq n-1} X_i$.
Denote by $\rho:\sttilt\Lambda\stackrel{\sim}{\to} \sttilt k\ORA{A_n}$ a poset isomorphism. From Lemma \ref{determiningarrow}, we observe $\rho(J_\Lambda)=J_{k\ORA{A_n}}$. By Lemma \ref{teqlem} (2), one also obtains $\rho(M_\Lambda)=M_{k\ORA{A_n}}$, since there is a path from $T_1$ to $0$ with length $n-1$.

Now, we show that $M_\Lambda=\bigoplus_{\ell=1}^{n-1}X(1,\ell)$.
The equality $\alpha_1\cdots\alpha_{n-1}=0$ implies that $P_1=X(1,n)=X(1,n-1)$, whence $T(n-1)=P_n\oplus\bigoplus_{\ell=2}^nX(1,\ell-1)$.
Therefore, we have an arrow $T(n-1)\to \bigoplus_{\ell=2}^nX(1,\ell-1)=:T(n)$.
Since $X_{r-1}$ is a projective $\Lambda/(e_r)$-module,
we see that $T(r)\leq T_r$.
Hence, one gets $T(r)=T(r-1)\wedge T_r\ (3\leq r\leq n)$ by the path (\ref{path1}) and an arrow $T(n-1)\to T(n)$,
which leads to the conclusion that $T(n)=T(2)\wedge T_{3}\wedge\cdots\wedge T_n=M_\Lambda$.

Since $T(n)$ belongs to $\sttilt_{P_n}\Lambda\simeq \sttilt\Lambda/(e_n)$, which has a maximum element $\Lambda/(e_n)$,
it finds out that $M_\Lambda=T(n)$ is less than  $\Lambda/(e_n)=J_\Lambda$.
As is checked above, we have $\rho(J_\Lambda)=J_{k\ORA{A_n}}$ and $\rho(M_\Lambda)=M_{k\ORA{A_n}}$, whence $M_{k\ORA{A_n}}\leq J_{k\ORA{A_n}}$.
However, this contradicts to Lemma \ref{teqlem} (1).
Thus, we obtain $\alpha_1\cdots \alpha_{n-1}\neq0$.
\end{proof}

In the rest of this paper, we deal with a finite dimensional algebra $\Lambda=kQ/I$ with $Q^\circ=\ORA{\T}$. Put $A:=k\ORA{\T}$.
%

Let us first describe relations of the quiver $Q$ when $\Lambda$ shares the support $\tau$-tilting poset structure with $A$

\begin{proposition}\label{casetree}
If $\sttilt\Lambda\simeq \sttilt k\ORA{\T}$, then the following hold:
\begin{enumerate}[(1)]
\item We have an equality $e_{i}\Lambda \alpha=\alpha\Lambda e_{j}$ for any arrow $\alpha:i\to j\ (i\neq j)$;
\item Every path of $\ORA{\T}$ does not belong to $I$.
\end{enumerate}
\end{proposition}
\begin{proof}
Let $\Lambda_{\alpha}:=\Lambda/(1-e_{i}-e_{j})$.
By Proposition \ref{mr} and  Lemma \ref{determiningarrow}, we obtain
\[e_{i}\Lambda \alpha =e_{i}\Lambda_{\alpha} \alpha = \alpha \Lambda_\alpha e_{j} = \alpha \Lambda e_{j}.\]

Let $w$ be a path $i_1\to\cdots\to i_{r}$ in $\ORA{\T}$. 
Applying Lemma \ref{determiningjoin} to $w$, we get an isomorphism
\[\sttilt\Lambda/(1-e_{i_{r}}-e_{i_{r-1}}-\cdots-e_{i_1})\simeq \sttilt k\ORA{A_r},\]
from which $w$ is not zero in $\Lambda/(1-e_{i_{r}}-e_{i_{r-1}}-\cdots-e_{i_1})$ by Corollary \ref{casea}.
Hence, one has $w\not\in I$.
%
%
\end{proof}

An algebra with two conditions as in Proposition \ref{casetree} is a split-by-nilpotent extension.

\begin{lemma}\label{fl}
Assume that the following two conditions are satisfied:
\begin{enumerate}[(i)]
\item $e_{i}\Lambda \alpha=\alpha\Lambda e_{j}$ for any arrow $\alpha:i\to j\ (i\neq j)$;
\item $w\not\in I$ for any path $w$ in $\ORA{\T}$.
\end{enumerate}
Then we have:
\begin{enumerate}[(1)]
\item $\Lambda$ is a split-by-nilpotent extension of $A:=k\ORA{\T}$.
\item This extension satisfies Condition \ref{CG}.  
\end{enumerate}
\end{lemma}
\begin{proof}
We show the assertion $(1)$.
As is mentioned in Proposition \ref{sbne}, it is sufficient to show that $I\subseteq L$, where $L$ is a two-sided ideal of $kQ$ generated by all loops on $Q$. Let $x\in I$.
Without loss of generality, we may assume that $x\in e_{i}Ie_{j}$. If there is no path from $i$ to $j$, then $x=0$.
Thus, we consider the case that there is a path from $i$ to $j$.
By $Q^{\circ}=\ORA{\T}$, we have a unique path 
\[w=w_{i,j}:i=i_1\xrightarrow{\alpha_1 } i_2\xrightarrow{\alpha_2}\cdots\xrightarrow{\alpha_{\ell-1}}i_\ell=j\]
from $i$ to $j$ which does not factor through any loop on $Q$.
Therefore, we can write $x=aw_{i,j}+y$ with $a\in k$ and $y\in L$.
As $y\in e_i\Lambda e_j$, one sees that $y$ is a linear combination of expressions $l_1\alpha_1 l_2\alpha_2\cdots l_{\ell-1}\alpha_{\ell-1}l_\ell$ with $l_t\in e_t\Rad (kQ)e_t$.
By the assumption (i), there exists $l'\in e_{i_t}\Rad(kQ) e_{i_t}$ with $l'\alpha_t=\alpha_tl_{t+1}$ in $\Lambda$.
Repeating this argument, we get $l\in e_i\Rad(kQ)e_i$ satisfying $y=lw_{i,j}$ in $\Lambda$.
Thus, one has $x=(ae_i+l)w_{i,j}\in I$.
This implies that $(a^2e_i-l^2)w_{i,j}=(ae_i-l)(ae_i+l)w_{i,j}$ also belongs to $I$, which yields that $a^Nw_{i,j}\in I$ since $\Rad^N kQ$ is contained in $I$ for sufficiently large $N>0$.
By the assumption (ii), we obtain $a=0$, whence $x=y\in L$.
   
We show the assertion $(2)$.
As the same argument above, it is observed that any element of $\Hom_\Lambda(e_j\Lambda, e_i\Lambda)\simeq e_i\Lambda e_j$ can be presented by a linear combination of $w_{i,j}l$ with $l\in \End_\Lambda(e_j\Lambda)\simeq e_j\Rad(kQ)e_j$.
Every homomorphism $f: e_jA\to e_iA$ admits a scalar $a$ such that $f(e_j)=aw_{i,j}$, which gives rise to a $\Lambda$-homomorphism $\widehat{f}:e_j\Lambda \xrightarrow{aw_{i,j}\cdot-}e_i\Lambda$.
Thus, Condition \ref{CG} (a) is satisfied.
%

Let $i\in Q_0$ and $l\in\End_\Lambda(\widehat{P_i})$.
For any neighbor vertex $j$ of $i$, we can take $l_j\in\End_\Lambda(\widehat{P_j})$ such that
\[\begin{cases}
\ l\circ \widehat{f_\alpha}=\widehat{f_\alpha}\circ l_j & \mbox{if } \alpha: i\to j\in \T_1; \\
\ \widehat{f_\alpha}\circ l=l_j\circ \widehat{f_\alpha} & \mbox{if } \alpha: j\to i\in \T_1.
\end{cases}\]
Note that $\T$ is a tree.
Inductively, we pick out endomorphisms $l_j\in\End_\Lambda(\widehat{P_j})$ for all $j\in\T_0$ satisfying $l_j\circ \widehat{f_\alpha}=\widehat{f_\alpha}\circ l_{j'}$ for any arrow $\alpha: j\to j'$.
This construction gives us endomorphisms as in Condition \ref{CG} (b).
%
\end{proof}


We now summarize our main theorems.

\begin{corollary}\label{fc}
The poset $\sttilt\Lambda$ is isomorphic to $\sttilt k\ORA{\T}$
 if and only if the following conditions hold:
 \begin{enumerate}[(a)]
 \item  $Q^{\circ}\simeq \ORA{\T}$;
 \item $e_{i}\Lambda \alpha=\alpha\Lambda e_{j}$
for any arrow $\alpha:i\to j\ (i\neq j)$;
\item  $w\not\in I$ for any path $w$ in $Q^\circ$.
\end{enumerate}
\end{corollary}
\begin{proof}
``if part'' By Lemma \ref{fl}, $\Lambda$ is a split-by-nilpotent extension of $A:=k\ORA{\T}$ satisfying Condition \ref{CG}.
As Theorem \ref{mainresult}, we get the desired poset isomorphism.

``only if part'' The three conditions follow from Remark \ref{remark} and Proposition \ref{casetree}.
\end{proof}

\begin{remark}
Independent of this paper, Eisele, Janssens and Raedschelders investigated a certain condition of ideals $I$ of an algebra $\Lambda$ to yield an isomorphism $\sttilt\Lambda\simeq \sttilt\Lambda/I$ \cite{EJR}:
They mentioned that such an isomorpshim is obtained if $I$ is contained in $(Z(\Lambda)\cap \Rad \Lambda)\cdot\Lambda$,
where $Z(\Lambda)$ denotes the center of $\Lambda$.
The ``if part'' of Corollary \ref{fc} can be also proved by using their result.

In fact, we construct an element $z$ of $\Lambda$ lying in $Z(\Lambda)\cap \Rad \Lambda$ as follows:

(i) For each vertex of $Q$, we define a positive integer $m_i$ as
the maximum of the integers $m$ satisfying $e_i(\Rad^m\Lambda)e_i\neq0$.

(ii) Fix a vertex $i$ of $Q$ at which $m_i$ takes the maximum among $m_j$ ($j\in Q_0$), and bring a non-zero element $l_i$ in $e_i(\Rad^{m_i}\Lambda)e_i$.

(iii) For a vertex $j$ of $Q$, $\ell_j$ denotes the length of a unique walk from $j$ to $i$ in $\T$: Note that $\T$ is a tree.
Let $\ell_j\geq1$.
Inductively, we define an element $l_j$ in $e_j\Lambda e_j$ as follows. Let $j'$ be a vertex of $Q$ with $\ell_{j'}=\ell_j-1$ and assume that there is an edge $\xymatrix{j \ar@{-}[r] & j'}$ in $\T$.
Then, by the condition (b) of Corollary \ref{fc} we have an element $l_j$ in $e_j\Lambda e_j$ such that 
\[\begin{cases}
\ l_{j'}\alpha=\alpha l_j & \mbox{if } \xymatrix{j' \ar[r]^\alpha & j}\ \mbox{in } \ORA{\T} \\
\ l_j\alpha=\alpha l_{j'} & \mbox{if } \xymatrix{j'  & j \ar[l]_\alpha}\ \mbox{in } \ORA{\T}.
\end{cases}\]
Note that $l_j$ also belongs to $e_j(\Rad^{m_i}\Lambda)e_j$.

(iv) We show that $z:=\sum_{j\in Q_0}l_j$ is in the center $Z(\Lambda)$. 
Let $l$ be a loop at a vertex $j$. Then, we have $lz=ll_j$,
which belongs to $e_j(\Rad^{m_i+1}\Lambda)e_j$.
Since $m_i$ is the maximum value, $lz$ finds out to be zero: Similarly, $zl=0$.
Let $\alpha: j\to j'$ be an arrow of $Q$. Then, we get $\alpha z=\alpha l_{j'}$ and $z\alpha=l_j \alpha$, which coincide by (iii).

Now, thanks to \cite[Theorem 4.1]{EJR}, we obtain an isomorphism $\sttilt\Lambda\xrightarrow{\sim}\sttilt\Lambda/(z)$.
Note that the three conditions (a), (b) and (c) of Corollary \ref{fc} are satisfied for $\Lambda/(z)$.

By repeating this argument, we realize $k\ORA{\T}$ as a quotient algebra of $\Lambda$ by some ideal $I$, and so one gets an isomorphism $\sttilt\Lambda\xrightarrow{\sim}\sttilt k\ORA{\T}$. 
%
%
%
\end{remark}

The following result is an immediate consequence of Theorem \ref{IRTT} and Corollary \ref{fc}, which gives an answer to Question \ref{questionlattice}.

\begin{corollary}
Let $\Lambda:=kQ/I$ be an algebra with $Q^\circ=\ORA{\T}$ and assume that $e_i\Lambda\alpha=\alpha\Lambda e_j$ and $w\not\in I$ for any arrow $\alpha:i\to j\ (i\neq j)$ and any path $w$ in $\ORA{\T}$. Then the poset $\sttilt\Lambda$ has lattice structure if and only if $\T$ is one of Dynkin diagrams ADE.
\end{corollary}

Finally, we give an example of algebras sharing the support $\tau$-tilting poset with tree quiver algebras.

Following to \cite[Section1]{GLS}, we construct a 1-Iwanaga-Gorenstein algebra $H$ from a symmetrizable generalized Cartan matrix $C$ and its orientation $\Omega$, that is, $C=(c_{ij})$ is a square integer matrix with order $n$ satisfying:
\begin{enumerate}[(C1)]
\item $c_{ii}=2$ for all $i$;
\item $c_{ij}\leq0$ for all $i\neq j$;
\item $c_{ij}\neq0$ if and only if $c_{ji}\neq0$;
\item There is a diagonal integer matrix $D=\mathrm{diag}(c_1,\cdots,c_n)$ with $c_i\geq1$ for all $i$ such that $DC$ is symmetric; such a $D$ is called a \emph{symmetrizer} of $C$.
\end{enumerate}
An orientation $\Omega$ of $C$ is a subset of $\{1,\cdots,n\}\times \{1,\cdots,n\}$ such that the following hold:
\begin{enumerate}[(O1)]
\item $\{(i,j), (j,i)\}\cap \Omega\neq\emptyset$ if and only if $c_{ij}<0$;
\item For any sequence $(i_1, i_2),(i_2, i_3),\cdots,(i_t,i_{t+1})$ of elements of $\Omega$ with $t\geq1$, we have $i_1\neq i_{t+1}$.
\end{enumerate}
It is evident that $(i,i)$ does not belong to $\Omega$.
We also notice that if $(i,j)$ is in $\Omega$, then $\Omega$ does not have $(j,i)$.

From the data of $C=(c_{ij}), D=\mathrm{diag}(c_1,\cdots,c_n)$ and $\Omega$, we define an algebra $H$.

\begin{definition}
For every $c_{ij}<0$, put $g_{ij}:=|\operatorname{gcd}(c_{ij}, c_{ji})|$ and $f_{ij}:=|c_{ij}|/g_{ij}$.

A quiver $Q:=Q(C, \Omega)$ is presented by
\begin{itemize}
\item $Q_0:=\{1,\cdots,n\}$;
\item $Q_1:=\{\alpha_{ij}^{(g)}: i\to j\ |\ (i,j)\in\Omega, 1\leq g\leq g_{ij} \}\cup \{\varepsilon_i:i\to i\ |\ 1\leq i\leq n \}$.
\end{itemize}
Note that $Q$ has precisely one loop at each vertex and no cycle.
We also observe that the underlying graph of $Q^\circ$ is a tree if and only if $c_{ij}\neq0\ (i\neq j)$ implies that $c_{ij}$ and $c_{ji}$ are coprime.

An algebra $H:=H(C,D,\Omega):=kQ/I$ is defined by
\begin{itemize}
\item $Q:=Q(C,\Omega)$;
\item $I$ is generated by
\begin{enumerate}[(H1)]
\item $\varepsilon_i^{c_i}$ for any $i$;
\item $\varepsilon_i^{f_{ij}}\alpha_{ij}^{(g)}-\alpha_{ij}^{(g)}\varepsilon_j^{f_{ji}}$ for any $(i,j)\in \Omega$ and any $g$ with $1\leq g \leq g_{ij}$.
\end{enumerate}
\end{itemize} 
\end{definition}

We have the following property of $H$.

\begin{theorem}\cite[Theorem 1.1]{GLS}
The algebra H is a finite dimensional 1-Iwanaga-Gorenstein algebra,
that is, the right and left injective dimensions of $H$ are 1.
\end{theorem}

It is also observed that the three conditions (a), (b) and (c) of Corollary \ref{fc} for $\Lambda=H$ hold if and only if $C$ is a symmetric matrix satisfying
\begin{itemize}
\item[(S)] For $i\neq j$, $c_{ij}\neq0$ implies $c_{ij}=-1$;
\end{itemize}

Thus, our goal is achieved.

\begin{proposition}
Let $C$ be a symmetrizable generalized Cartan matrix with symmetrizer $D$ and $\Omega$ its orientation.
Put $Q:=Q(C,\Omega)$ and $H:=H(C,D,\Omega)$.
If $C$ is symmetric and satisfies the condition (S), then we have a poset isomorphism $\sttilt H\xrightarrow{\sim}\sttilt kQ^\circ$.
\end{proposition}


\end{document}